\documentclass[11pt]{article}

\usepackage{amsthm}
\usepackage{latexsym}
\usepackage{dsfont}
\usepackage{bbm}
\usepackage{amssymb}
\usepackage{amsmath}
\usepackage{graphicx}
\usepackage{color}
\usepackage{wasysym}
\usepackage[noblocks,affil-it]{authblk}

\numberwithin{equation}{section}

\theoremstyle{plain}
\newtheorem{Theorem}{Theorem}[section]
\newtheorem{Lemma}[Theorem]{Lemma}

\theoremstyle{remark}

\theoremstyle{definition}

\DeclareMathOperator{\N}{\mathbb{N}}

\DeclareMathOperator{\R}{\mathbb{R}}

\newcommand{\od}{\overset{{\rm d}}{=}}

\newcommand{\tp}{\overset{\mathbb{P}}{\to }}
\newcommand{\mn}{\mathbb{N}}
\newcommand{\me}{\mathbb{E}}
\newcommand{\lit}{\lim_{t\to\infty}}

\DeclareMathOperator{\Prob}{\mathbb{P}}

\DeclareMathOperator{\E}{\mathbb{E}}

\DeclareMathOperator{\1}{\mathbbm{1}}

\title{Functional limit theorems for the number of busy servers in a $G/G/\infty$ queue}
\date{\today}

\author{Alexander Iksanov\footnote{Faculty of Computer Science and Cybernetics, Taras
Shevchenko National University of Kyiv, 01601 Kyiv, Ukraine and
Institute of Mathematics, University of Wroc{\l}aw, 50-384
Wroc{\l}aw, Poland; e-mail: iksan@univ.kiev.ua}, \ Wissem Jedidi
\footnote{Department of Statistics  \& OR, King Saud University,
P.O. Box 2455, Riyadh 11451, Saudi Arabia and Universit\'e de
Tunis El Manar, Facult\'e des Sciences de Tunis, LR11ES11
Laboratoire d'Analyse Math\'ematiques et Applications, 2092,
Tunis, Tunisia; e-mail: wissem$_-$jedidi@yahoo.fr} \ and \ Fethi
Bouzeffour \footnote{Department of Mathematics, College of
Sciences, King Saud University, Riyadh 11451, Saudi Arabia;
e-mail: fbouzaffour@ksu.edu.sa}}

\begin{document}
\maketitle
\begin{abstract}
We discuss weak convergence of the number of busy servers in a
$G/G/\infty$ queue in the $J_1$-topology on the Skorokhod space.
We prove two functional limit theorems, with random and nonrandom
centering, respectively, thereby solving two open problems stated
in \cite{Mikosch+Resnick:2006}. A new integral representation for
the limit Gaussian process is given.
\end{abstract}


\section{Introduction}

Let $(\xi_k, \eta_k)_{k\in\mn}$ be a sequence of i.i.d.\
two-dimensional random vectors with generic copy $(\xi,\eta)$
where both $\xi$ and $\eta$ are positive. No condition is imposed
on the dependence structure between $\xi$ and $\eta$.

Define
\begin{equation*}
K(t)~:=~\sum_{k\geq 0}\1_{\{S_k+\eta_{k+1}\leq
t\}}\quad\text{and}\quad Z(t) ~:=~ \sum_{k\geq 0}\1_{\{S_k\leq
t<S_k+\eta_{k+1}\}}, \quad t \geq 0,
\end{equation*}
where\footnote{$\mn_0:=\mn\cup\{0\}$.} $(S_k)_{k\in\mn_0}$ is the
zero-delayed ordinary random walk with increments $\xi_k$ for
$k\in\mn$, i.e., $S_0 = 0$ and $S_k = \xi_1+\ldots+\xi_k$, $k \in
\mn$. In a ${\rm G}/{\rm G}/\infty$-queuing system, where
customers arrive at times $S_0 = 0 < S_1 < S_2 < \ldots$ and are
immediately served by one of infinitely many idle servers, the
service time of the $k$th customer being $\eta_{k+1}$, $K(t)$
gives the number of customers served up to and including time
$t\geq 0$, whereas $Z(t)$ gives the number of busy servers at time
$t$. Some other interpretations of $Z(t)$ can be found in
\cite{Kaplan:1975}. The process $(Z(t))_{t \geq 0}$ was also used
to model the number of active sources in a communication network
(for instance, active sessions in a computer network)
\cite{Konstantopoulos+Lin:1998, Mikosch+Resnick:2006,
Resnick+Rootzen:2000}.

From a more theoretical viewpoint, $K(t)$ is the number of visits
to the interval $[0,t]$ of a {\it perturbed random walk}
$(S_k+\eta_{k+1})_{k\in\mn_0}$ and $Z(t)$ is the difference
between the number of visits to $[0,t]$ of the ordinary random
walk $(S_k)_{k\in\mn_0}$ and $(S_k+\eta_{k+1})_{k\in\mn_0}$. To
proceed, we need a definition. Denote by $X:=(X(t))_{t\geq 0}$ a
random process arbitrarily dependent on $\xi$. Let
$(X_k,\xi_k)_{k\in\mn}$ be i.i.d.\ copies of the pair $(X,\xi)$.
Following \cite{Iksanov+Marynych+Meiners:2016} we call {\it random
process with immigration} the random process $(Y(t))_{t\geq 0}$
defined by
$$Y(t):=\sum_{k\geq 0}X_{k+1}(t-S_k)\1_{\{S_k\leq t\}},\quad
t\geq 0.$$ If $X$ is deterministic, the random process with
immigration becomes a classical renewal shot noise process.
Getting back to the mainstream we conclude that both
$(K(t))_{t\geq 0}$ and $Z:=(Z(t))_{t\geq 0}$ are particular
instances of the random process with immigration which correspond
to $X(t)=\1_{\{\eta\leq t\}}$ and $X(t)=\1_{\{\eta>t\}}$,
respectively.

Let $D:=D[0,\infty)$ be the Skorokhod\footnote{The Skorokhod
spaces $D(0,\infty)$ and $D[0,T]$ for $T>0$ which appear below are
defined similarly.} space of real-valued functions on
$[0,\infty)$, which are right-continuous on $[0,\infty)$ with
finite limits from the left at each positive point. We shall write
$\overset{{\rm J_1}}{\Rightarrow}$ and $\tp$ to denote weak
convergence in the $J_1$-topology on $D$ and convergence in
probability, respectively. The classical references concerning the
$J_1$-topology are \cite{Billingsley:1968, Jacod+Shiryaev:2003,
Lindvall:1973}.

In this paper we shall prove weak convergence of $(Z(ut))_{u\geq
0}$, properly centered and normalized, in the $J_1$-topology on
$D$ as $t\to\infty$. The same problem for $(K(ut))_{u\geq 0}$
which is much simpler was solved in
\cite{Alsmeyer+Iksanov+Marynych:2016}. We start with a functional
limit theorem with a random centering.
\begin{Theorem}\label{main1}
Assume that $\mu:=\me\xi\in (0,\infty)$ and that
\begin{equation}\label{reg}
1-F(t)=\Prob\{\eta>t\}\sim t^{-\beta}\ell(t),\quad t\to\infty
\end{equation}
for some $\beta\in [0,1)$ and some $\ell$ slowly varying at
$\infty$. Then
\begin{equation}\label{180018z}
\frac{\sum_{k\geq 0}\big(\1_{\{S_k \leq ut <
S_k+\eta_{k+1}\}}-(1-F(ut-S_k))\1_{\{S_k\leq ut\}}\big)
}{\sqrt{\mu^{-1}\int_0^t (1-F(y))\,{\rm d}y}}~\overset{{\rm
J_1}}{\Rightarrow}~ V_\beta(u),\quad t\to\infty,
\end{equation}
where $V_\beta:=(V_\beta(u))_{u\ge 0}$ is a centered Gaussian
process with
\begin{equation}\label{covar}
\me\, V_\beta(u)V_\beta(s)=u^{1-\beta}-(u-s)^{1-\beta}, \quad 0\le
s\le u.
\end{equation}
\end{Theorem}

In the case where $\xi$ and $\eta$ are independent weak
convergence of the finite-dimensional distributions in
\eqref{180018z} was proved in Proposition 3.2 of
\cite{Mikosch+Resnick:2006}. In the general case treated here
where $\xi$ and $\eta$ are arbitrarily dependent the
aforementioned convergence outside zero (i.e., weak convergence of
$(Z^\ast_t(u_1),\ldots, Z^\ast_t(u_n))$ for any $n\in\mn$ and any
$0<u_1<\ldots<u_n<\infty$, where $Z^\ast_t(u)$ denotes the
left-hand side in \eqref{180018z}) follows from a specialization
of Proposition 2.1 in \cite{Iksanov+Marynych+Meiners:2016}. In
Section 5.2 of \cite{Mikosch+Resnick:2006} the authors write: `We
suspect that the' finite-dimensional `convergence can be
considerably strengthened'. Our Proposition \ref{main1} confirms
their conjecture.

Also, the authors of \cite{Mikosch+Resnick:2006} ask on p.~154:
`When can the random centering' in \eqref{180018z} `be replaced by
a non-random centering?' Our second result states that such a
replacement is possible whenever $\xi$ possesses finite moments of
sufficiently large positive orders. Our approach is essentially
based on the decomposition\footnote{Investigating $Z$ directly,
i.e., not using \eqref{deco}, seems to be a formidable task unless
$\xi$ and $\eta$ are independent, and the distribution of $\xi$ is
exponential (for the latter situation, see
\cite{Resnick+Rootzen:2000} and references therein). We note in
passing that our Theorem \ref{main} includes Theorem 1 in
\cite{Resnick+Rootzen:2000} as a particular case.}
\begin{eqnarray}
&&\sum_{k\geq 0}\1_{\{S_k\leq
ut<S_k+\eta_{k+1}\}}-\mu^{-1}\int_0^{ut}(1-F(y))\,{\rm
d}y\notag\\&=& \bigg(\sum_{k\geq 0}\1_{\{S_k\leq
ut<S_k+\eta_{k+1}\}}-\sum_{k\geq 0}\me (\1_{\{S_k\leq
ut<S_k+\eta_{k+1}\}}|S_k)\bigg)\notag\\&+&\bigg(\sum_{k\geq 0}\me
(\1_{\{S_k\leq ut<S_k+\eta_{k+1}\}}|S_k)-\mu^{-1}\int_0^{ut}
(1-F(y))\,{\rm d}y\bigg)\notag\\&=&\sum_{k\geq
0}\bigg(\1_{\{S_k\leq ut<S_k+\eta_{k+1}\}}-(1-F(ut-S_k))
\1_{\{S_k\leq ut\}}\bigg)\notag\\&+&\bigg(\sum_{k\geq
0}(1-F(ut-S_k))\1_{\{S_k\leq ut\}}-\mu^{-1}\int_0^{ut}
(1-F(y))\,{\rm d}y\bigg).\label{deco}
\end{eqnarray}
Weak convergence on $D$ of the first summand on the right-hand
side, normalized by $\sqrt{\mu^{-1}\int_0^t (1-F(y))\,{\rm d}y}$,
was treated in Theorem \ref{main1}. Thus, we are left with
analyzing weak convergence of the second summand.
\begin{Theorem}\label{main}
Suppose that condition \eqref{reg} holds. If $\me \xi^r<\infty$
for some $r>2(1-\beta)^{-1}$, then
\begin{equation}\label{180018}
\frac{\sum_{k\geq 0}\1_{\{S_k \leq ut <
S_k+\eta_{k+1}\}}-\mu^{-1}\int_0^{ut} (1-F(y))\,{\rm
d}y}{\sqrt{\mu^{-1}\int_0^t (1-F(y))\,{\rm d}y}}~\overset{{\rm
J_1}}{\Rightarrow}~ V_\beta(u),\quad t\to\infty,
\end{equation}
where $\mu=\me\xi<\infty$ and $V_\beta$ is a centered Gaussian
process with covariance \eqref{covar}.
\end{Theorem}

Under the assumption that $\xi$ and $\eta$ are independent weak
convergence of the {\it one-dimensional} distributions in
\eqref{180018} was proved in Theorem 2 of \cite{Kaplan:1975}. Note
that regular variation condition \eqref{reg} is not needed for
this convergence to hold. Weak convergence of the {\it
finite-dimensional} distributions in \eqref{180018} takes place
under \eqref{reg} and the weaker assumption $\me\xi^2<\infty$. We
do not know whether \eqref{reg} and the second moment assumption
are sufficient for weak convergence on $D$. More generally, weak
convergence of the finite-dimensional distributions of $Z(ut)$,
properly\footnote{The normalization is not necessarily of the form
$\sqrt{\mu^{-1}\int_0^t (1-F(y))\,{\rm d}y}$, and the limit
process is not necessarily $V_\beta$.} normalized and centered,
holds whenever the distribution of $\xi$ belongs to the domain of
attraction of an $\alpha$-stable distribution, $\alpha\in
(0,2]\backslash\{1\}$, see Theorem 3.3.21 in \cite{Iksanov:2016}
which is a specialization of Theorems 2.4 and 2.5 in
\cite{Iksanov+Marynych+Meiners:2016}. We do not state these
results here because in this paper we are only interested in weak
convergence on $D$.

The rest of the paper is structured as follows. Theorems
\ref{main1} and \ref{main} are proved in Sections \ref{pr1} and
\ref{pr2}, respectively. In Section \ref{proc} we discuss an
integral representation of the limit process $V_\beta$ which seems
to be new. The appendix collects several auxiliary results.

\section{Proof of Theorem \ref{main1}}\label{pr1}

We start by observing that
\begin{equation}\label{regular}
a(t):=\sum_{k=0}^{[t]+1}(1-F(k))\sim \int_0^t (1-F(y)){\rm d}y
\sim (1-\beta)^{-1}t^{1-\beta}\ell (t)
\end{equation}
as $t\to\infty$, where the second equivalence follows from
Karamata's theorem (Proposition 1.5.8 in \cite{BGT}). In
particular, the first equivalence enables us to replace the
integral in the denominator of \eqref{180018z} with the sum. For
each $t,u\geq 0$, denote by $\widehat{Z}(ut)$ the first summand in
decomposition \eqref{deco}, i.e.,
\begin{eqnarray*}
\widehat{Z}(ut)&:=&\sum_{k\geq 0}\big(\1_{\{S_k\leq
ut<S_k+\eta_{k+1}\}}-(1-F(ut-S_k))\1_{\{S_k\leq
ut\}}\big)\\&=&\sum_{k\geq 0}\big(\1_{\{S_k+\eta_{k+1}\leq
ut\}}-F(ut-S_k)\1_{\{S_k\leq ut\}}\big)
\end{eqnarray*}
and then set
$$Z_t(u):=\frac{\sum_{k\geq 0}\big(\1_{\{S_k \leq ut <
S_k+\eta_{k+1}\}}-(1-F(ut-S_k))\1_{\{S_k\leq ut\}}\big)
}{\sqrt{a(t)}}=\frac{\widehat{Z}(ut)}{\sqrt{a(t)}},\quad u\geq
0.$$

Our proof of Theorem \ref{main1} is similar to the proof of
Theorem 1 in \cite{Resnick+Rootzen:2000} which treats the case
where $\xi$ and $\eta$ are independent, and the distribution of
$\xi$ is exponential (Poisson case). Lemma \ref{imp} given below
is concerned with inevitable technical complications that appear
outside the Poisson case. Put $$\nu(t):=\inf\{k\in\mn_0:
S_k>t\},\quad t\in\R$$ and note that the random variable $\nu(1)$
has finite moments of all positive orders by Lemma \ref{fin}.
\begin{Lemma}\label{imp}
Let $l\in\mn$ and $0\leq v<u$. For any chosen $A>1$ and $\rho\in
(0,1-\beta)$ there exist $t_1>1$ such that $$\me
|Z_t(u)-Z_t(v)|^{2l}\leq c(l)(u-v)^{l(1-\beta-\rho)}$$ whenever
$u-v<1$ and $(u-v)t\geq t_1$, where
$c(l):=2C_l(4A)^l(u-v)^{l(1-\beta-\rho)}\me (\nu(1))^l$ and $C_l$
is a finite positive constant.
\end{Lemma}
\begin{proof}
With $u,v\geq 0$ fixed, $\widehat{Z}(ut)-\widehat{Z}(vt)$ equals the terminal value of
the martingale $(R(k,t),\mathcal{F}_k)_{k\in\mn_0}$, where
$R(0,t):=0$,
\begin{eqnarray*}
R(k,t)&:=&\sum_{j=0}^{k-1}\big((\1_{\{S_j+\eta_{j+1}\leq
ut\}}-F(ut-S_j)\1_{\{S_j\leq ut\}})\\&-&(\1_{\{S_j+\eta_{j+1}\leq
vt\}}-F(vt-S_j)\1_{\{S_j\leq vt\}})\big),
\end{eqnarray*}
$\mathcal{F}_0:=\{\Omega, \oslash\}$ and
$\mathcal{F}_k:=\sigma((\xi_j,\eta_j):1\leq j\leq k)$. We use the
Burkholder-Davis-Gundy inequality (Theorem 11.3.2 in
\cite{Chow+Teicher:2003}) to obtain for any $l\in\mn$
\begin{eqnarray}
&&\E (\widehat{Z}(ut)-\widehat{Z}(vt))^{2l}\notag\\
&\leq& C_l\bigg(\E\bigg(\sum_{k\geq 0}\E
\big((R(k+1,t)-R(k,t))^2|\mathcal{F}_k\big)\bigg)^l+\sum_{k\geq
0}\E
\big(R(k+1,t)-R(k,t)\big)^{2l}\bigg)\notag\\
&=:&C_l(I_1(t)+I_2(t))\label{mom}
\end{eqnarray}
for a positive constant $C_l$. We shall show that
\begin{equation}\label{11}
I_1(t)\leq 2^l \me (\nu(1))^l(a((u-v)t))^l,\quad t\geq 0
\end{equation}
and that
\begin{equation}\label{22}
I_2(t)\leq 2^{2l}\me \nu(1)a((u-v)t),\quad t\geq 0.
\end{equation}

\noindent {\sc Proof of \eqref{11}}. We first observe that
\begin{eqnarray*}
&&\sum_{k\geq 0}\E
\big((R(k+1,t)-R(k,t))^2|\mathcal{F}_k\big)\\&=&\int_{(vt,\,ut]}F(ut-y)(1-F(ut-y)){\rm
d}\nu(y)\\&+&\int_{[0,\,vt]}(F(ut-y)-F(vt-y))(1-F(ut-y)+F(vt-y)){\rm
d}\nu(y)\\&\leq& \int_{(vt,\,ut]}(1-F(ut-y)){\rm
d}\nu(y)+\int_{[0,\,vt]}(F(ut-y)-F(vt-y)){\rm d}\nu(y)
\end{eqnarray*}
whence
\begin{eqnarray*}
I_1(t)&\leq& 2^{l-1}\bigg(\me
\bigg(\int_{(vt,\,ut]}(1-F(ut-y)){\rm
d}\nu(y)\bigg)^l\\&+&\me\bigg(\int_{[0,\,vt]}(F(ut-y)-F(vt-y)){\rm
d}\nu(y)\bigg)^l\bigg)
\end{eqnarray*}
having utilized $(x+y)^l\leq 2^{l-1}(x^l+y^l)$ for nonnegative $x$
and $y$. Using Lemma \ref{impo1} with
$G(y)=(1-F(y))\1_{[0,(u-v)t)}(y)$ and $G(y)=F((u-v)t+y)-F(y)$,
respectively, we obtain
\begin{eqnarray}
&&\me \bigg(\int_{(vt,\,ut]}(1-F(ut-y)){\rm
d}\nu(y)\bigg)^l\notag\\&=&\me
\bigg(\int_{[0,\,ut]}(1-F(ut-y))\1_{[0,\,(u-v)t)}(ut-y){\rm
d}\nu(y)\bigg)^l\notag\\&\leq& \me (\nu(1))^l
\bigg(\sum_{n=0}^{[ut]}\sup_{y\in
[n,\,n+1)}((1-F(y))\1_{[0,(u-v)t)}(y))\bigg)^l\notag\\&\leq& \me
(\nu(1))^l \bigg(\sum_{n=0}^{[(u-v)t]}(1-F(n))\bigg)^l\leq \me
(\nu(1))^l (a((u-v)t))^l .\label{11a}
\end{eqnarray}
and
\begin{eqnarray}
&&\me\bigg(\int_{[0,\,vt]}(F(ut-y)-F(vt-y)){\rm
d}\nu(y)\bigg)^l\notag\\&\leq& \me (\nu(1))^l
\bigg(\sum_{n=0}^{[vt]}\sup_{y\in
[n,\,n+1)}(F((u-v)t+y)-F(y))\bigg)^l\notag\\&\leq& \me (\nu(1))^l
\bigg(\sum_{n=0}^{[vt]}(1-F(n))-\sum_{n=0}^{[vt]}(1-F((u-v)t+n+1))\bigg)^l \notag\\
&\leq& \me (\nu(1))^l
\bigg(\sum_{n=0}^{[vt]}(1-F(n))-\sum_{n=0}^{[ut]+2}(1-F(n))+\sum_{n=0}^{[(u-v)t]+1}(1-F(n))\bigg)^l\notag\\&\leq&
\me(\nu(1))^l (a((u-v)t))^l.\label{11b}
\end{eqnarray}
Combining \eqref{11a} and \eqref{11b} yields \eqref{11}.

\noindent {\sc Proof of \eqref{22}}. Let us calculate
\begin{eqnarray*}
&&\me ((R(k+1,t)-R(k,t))^{2l}|\mathcal{F}_k)\\&\leq
&2^{2l-1}((1-F(ut-S_k))^{2l}F(ut-S_k)\\&+&(F(ut-S_k))^{2l}(1-F(ut-S_k)))\1_{\{vt<S_k\leq
ut\}}\\&+&((1-F(ut-S_k)+F(vt-S_k))^{2l}(F(ut-S_k)-F(vt-S_k))\\&+&(F(ut-S_k)-F(vt-S_k))^{2l}(1-F(ut-S_k)+F(vt-S_k)))\1_{\{S_k\leq
vt\}}\\&\leq& 2^{2l-1}((1-F(ut-S_k))\1_{\{vt<S_k\leq
ut\}}+(F(ut-S_k)-F(vt-S_k))\1_{\{S_k\leq vt\}}).
\end{eqnarray*}
Therefore, $$I_2(t)\leq 2^{2l-1}\bigg(\me
\int_{(vt,\,ut]}(1-F(ut-y)){\rm d}\nu(y)+\me
\int_{[0,\,vt]}(F(ut-y)-F(vt-y)){\rm d}\nu(y)\bigg).$$ Using now
formulae \eqref{11a} and \eqref{11b} with $l=1$ yields \eqref{22}.

In view of \eqref{regular} we can invoke Potter's bound (Theorem
1.5.6(iii) in \cite{BGT}) to conclude that for any chosen $A>1$
and $\rho\in (0,1-\beta)$ there exists $t_1>1$ such that
$$\frac{a((u-v)t)}{a(t)}\leq A(u-v)^{1-\beta-\rho}$$
whenever $u-v<1$ and $(u-v)t\geq t_1$. Note that $u-v<1$ and
$(u-v)t\geq t_1$ together imply $t\geq t_1$. Hence
\begin{eqnarray}
\frac{I_1(t)}{(a(t))^l}&\leq& 2^l\me
(\nu(1))^l\bigg(\frac{a((u-v)t)}{a(t)}\bigg)^l \leq (4A)^l\me
(\nu(1))^l (u-v)^{l(1-\beta-\rho)}.\label{aux1}
\end{eqnarray}
Increasing $t_1$ if needed we can assume that
$t^{1-\beta-\rho}/a(t)\leq 1$ for $t\geq t_1$ whence
\begin{eqnarray*}
\frac{1}{\sum_{n=0}^{[t]+1}(1-F(n))}&=&\frac{((u-v)t)^{1-\beta-\rho}}{((u-v)t)^{1-\beta-\rho}
\sum_{n=0}^{[t]+1}(1-F(n))}\\&\leq&
\frac{(u-v)^{1-\beta-\rho}}{((u-v)t)^{1-\beta-\rho}}\leq
(u-v)^{1-\beta-\rho}
\end{eqnarray*}
because $((u-v)t)^{1-\beta-\rho}\geq t_1^{1-\beta-\rho}>1$. This
implies
\begin{eqnarray}
\frac{I_2(t)}{(a(t))^l} &\leq& 2^{2l}\me
\nu(1)\frac{a((u-v)t)}{a(t)} \frac{1}{(a(t))^{l-1}} \leq (4A)^l
\me (\nu(1))^l (u-v)^{l(1-\beta-\rho)},\label{aux2}
\end{eqnarray}
where we have used $\me \nu(1)\leq \me (\nu(1))^l$ which is a
consequence of $\nu(1)\geq 1$ a.s.

Now the claim follows from \eqref{mom}, \eqref{aux1} and
\eqref{aux2}.
\end{proof}

We are ready to prove Theorem \ref{main1}. As discussed in the
paragraph following Theorem \ref{main1} weak convergence of
$(Z_t(u_1),\ldots, Z_t(u_n))$ for any $n\in\mn$ and any
$0<u_1<\ldots, u_n<\infty$ was proved in earlier works. In view of
$V_\beta(0)=0$ a.s., this immediately extends to $0\leq
u_1<\ldots, u_n<\infty$. Thus, it remains to prove tightness on
$D[0,T]$ for any $T>0$. Since the normalization in \eqref{180018z}
is regularly varying it is enough to investigate the case $T=1$
only. Suppose we can prove that for any $\varepsilon>0$ and
$\gamma>0$ there exist $t_0>0$ and $\delta>0$ such that
\begin{equation}\label{tight} \Prob\big\{\sup_{0\leq u,v\leq
1, |u-v|\leq \delta}|Z_t(u)-Z_t(v)|>\varepsilon\big\}\leq \gamma
\end{equation}
for all $t\geq t_0$. Then, by Theorem 15.5 in
\cite{Billingsley:1968} the desired tightness follows along with
continuity of the paths of (some version of) the limit process.

On pp.~763-764 in \cite{Resnick+Rootzen:2000} it is shown that
(the specific form of $Z_t$ plays no role here)
\begin{eqnarray*}
\sup_{0\leq u,v\leq 1, |u-v|\leq 2^{-i}}|Z_t(u)-Z_t(v)|&\leq&
2\sum_{j=i}^I\max_{1\leq k\leq
2^j}|Z_t(k2^{-j})-Z_t((k-1)2^{-j})|\\&+& 2\max_{0\leq k\leq
2^{I}-1}\sup_{0\leq w\leq 2^{-I}}|Z_t(k2^{-I}+w)-Z_t(k2^{-I})|
\end{eqnarray*}
for any positive integers $i$ and $I$, $i\leq I$. Hence
\eqref{tight} follows if we can check that for any $\varepsilon>0$
and $\gamma>0$ there exist $t_0>0$, $i\in\mn$ and $I\in\mn$,
$i\leq I$ such that
\begin{equation}\label{tight1}
\Prob\bigg\{\sum_{j=i}^I\max_{1\leq k\leq
2^j}|Z_t(k2^{-j})-Z_t((k-1)2^{-j})|>\varepsilon\bigg\}\leq
\gamma,\quad t\geq t_0
\end{equation}
and that
\begin{equation}\label{tight2}
\max_{0\leq k\leq 2^{I}-1}\sup_{0\leq w\leq
2^{-I}}|Z_t(k2^{-I}+w)-Z_t(k2^{-I})|~\tp~ 0,\quad t\to\infty.
\end{equation}

\noindent {\sc Proof of \eqref{tight1}}. By Lemma \ref{imp}, for
any chosen $A>1$ and $\rho\in (0, 1-\beta)$ there exists $t_1>1$
such that
\begin{equation}\label{ineq}
\me |Z_t(k2^{-j})-Z_t((k-1)2^{-j})|^{2l}\leq
c(l)2^{-jl(1-\beta-\rho)}
\end{equation}
whenever $2^{-j}t\geq t_1$. Let $I=I(t)$ denote the integer number
satisfying
$$2^{-I}t\geq t_1>2^{-I-1}t.$$ Then the inequalities \eqref{ineq} and
\begin{eqnarray*}
\me (\max_{1\leq k\leq
2^j}|Z_t(k2^{-j})-Z_t((k-1)2^{-j})|)^{2l}&\leq&
\sum_{k=1}^{2^j}\me |Z_t(k2^{-j})-Z_t((k-1)2^{-j})|^{2l}\\&\leq&
c(l)2^{-j(l(1-\beta-\rho)-1)}
\end{eqnarray*}
hold whenever $j\leq I$. Pick now minimal $l\in\mn$ such that
$l(1-\beta-\rho)>1$. Given positive $\varepsilon$ and $\gamma$
choose minimal $i\in\mn$ satisfying
$$2^{-i(l(1-\beta-\rho)-1)}\leq \varepsilon^{2l}
(1-2^{-(l(1-\beta-\rho)-1)/(2l)})^{2l}\gamma/ c(l).$$ Increase $t$
if needed to ensure that $i\leq I$. Invoking Markov's inequality
and then the triangle inequality for the $L_{2l}$-norm gives
\begin{eqnarray*}
&&\Prob\{\sum_{j=i}^I \max_{1\leq k\leq
2^j}|Z_t(k2^{-j})-Z_t((k-1)2^{-j})|>\varepsilon\}\\&\leq&
\varepsilon^{-2l}\me \bigg(\sum_{j=i}^I \max_{1\leq k\leq
2^j}|Z_t(k2^{-j})-Z_t((k-1)2^{-j})|\bigg)^{2l}\\&\leq&
\varepsilon^{-2l}\bigg(\sum_{j=i}^I (\me (\max_{1\leq k\leq
2^j}|Z_t(k2^{-j})-Z_t((k-1)2^{-j})|)^{2l})^{1/2l}\bigg)^{2l}\\&\leq&
\varepsilon^{-2l} c(l)\bigg(\sum_{j\geq
i}2^{-j(l(1-\beta-\rho)-1)/(2l)}\bigg)^{2l}\\&=&\varepsilon^{-2l}
c(l)\frac{2^{-i(l(1-\beta-\rho)-1)}}{(1-2^{-(l(1-\beta-\rho)-1)/(2l)})^{2l}}
\leq \gamma
\end{eqnarray*}
for all $t$ large enough.

\noindent {\sc Proof of \eqref{tight2}}. We shall use a
decomposition
\begin{eqnarray*}
&&(a(t))^{1/2}(Z_t(k2^{-I}+w)-Z_t(k2^{-I}))\\&=&\sum_{j\geq
0}\big(\1_{\{S_j+\eta_{j+1}\leq
(k2^{-I}+w)t\}}-F((k2^{-I}+w)t-S_j)\big)\1_{\{k2^{-I}t<S_j\leq
(k2^{-I}+w)t\}}\\&+& \sum_{j\geq
0}\big(\1_{\{k2^{-I}t<S_j+\eta_{j+1}\leq
(k2^{-I}+w)t\}}\\&-&(F((k2^{-I}+w)t-S_j)-F(k2^{-I}t-S_j))\big)\1_{\{S_j\leq
k2^{-I}t\}} \\&=:&J_1(t,k,w)+J_2(t,k,w).
\end{eqnarray*}
It suffices to prove that for $i=1,2$
\begin{equation}\label{33}
(a(t))^{-1/2}\max_{0\leq k\leq 2^I-1}\sup_{0\leq w\leq
2^{-I}}|J_i(t,k,w)|~\tp~ 0,\quad t\to\infty.
\end{equation}

\noindent {\sc Proof of \eqref{33} for $i=1$}. Since
$|J_1(t,k,w)|\leq \nu((k2^{-I}+w)t)-\nu(k2^{-I}t)$ and $\nu(t)$ is
a.s.\ nondecreasing we infer $\sup_{0\leq w\leq
2^{-I}}|J_1(t,k,w)|\leq \nu ((k+1)2^{-I}t)-\nu(k2^{-I}t)$. By
Boole's inequality and distributional subadditivity of $\nu(t)$
(see formula (5.7) on p.~58 in \cite{Gut:2009})
\begin{eqnarray*}
&&\Prob\big\{\max_{0\leq k\leq 2^{I}-1}(\nu
((k+1)2^{-I}t)-\nu(k2^{-I}t))>\delta (a(t))^{1/2}\big\}\\&\leq&
\sum_{k=0}^{2^I-1}\Prob\big\{\nu
((k+1)2^{-I}t)-\nu(k2^{-I}t)>\delta (a(t))^{1/2}\big\}\\&\leq& 2^I
\Prob\big\{\nu(2^{-I}t)>\delta (a(t))^{1/2}\big\}\leq 2^I
\Prob\big\{\nu(2t_1)>\delta (a(t))^{1/2}\big\}
\end{eqnarray*}
for any $\delta>0$. The right-hand side converges to zero as
$t\to\infty$ because $\nu(2t_1)$ has finite exponential moments of
all positive orders (see Lemma \ref{fin}).

\noindent {\sc Proof of \eqref{33} for $i=2$}. We have
\begin{eqnarray*}
&&\sup_{0\leq w\leq 2^{-I}}|J_2(t,k,w)|\\&\leq& \sup_{0\leq w\leq
2^{-I}} \bigg(\sum_{j\geq 0}\1_{\{k2^{-I}t<S_j+\eta_{j+1}\leq
(k2^{-I}+w)t\}}\1_{\{S_j\leq k2^{-I}t\}}\\&+&\sum_{j\geq
0}(F((k2^{-I}+w)t-S_j)-F(k2^{-I}t-S_j))\big)\1_{\{S_j\leq
k2^{-I}t\}}\bigg)\\&\leq& \sum_{j\geq
0}\1_{\{k2^{-I}t<S_j+\eta_{j+1}\leq (k+1)2^{-I}t\}}\1_{\{S_j\leq
k2^{-I}t\}}\\&+&\sum_{j\geq
0}(F(((k+1)2^{-I})t-S_j)-F(k2^{-I}t-S_j))\big)\1_{\{S_j\leq
k2^{-I}t\}}\bigg)\\&\leq& \bigg|\sum_{j\geq
0}\big(\1_{\{k2^{-I}t<S_j+\eta_{j+1}\leq
(k+1)2^{-I}t\}}\\&-&(F(((k+1)2^{-I})t-S_j)-F(k2^{-I}t-S_j))\big)\1_{\{S_j\leq
k2^{-I}t\}}\bigg|\\&+& 2\sum_{j\geq
0}(F(((k+1)2^{-I})t-S_j)-F(k2^{-I}t-S_j))\big)\1_{\{S_j\leq
k2^{-I}t\}}\\&=:&J_{21}(t,k)+2J_{22}(t,k).
\end{eqnarray*}

Pick minimal $r\in\mn$ satisfying $r(1-\beta)>1$ so that
$\lim_{t\to\infty} t^{-1}(a(t))^r=\infty$. Using \eqref{11b} with
$u=(k+1)2^{-I}$ and $v=k2^{-I}$ we obtain $$\me
(J_{22}(t,k))^{2r}\leq \me (\nu(1))^{2r} (a(2^{-I}t))^{2r}\leq \me
(\nu(1))^{2r} (a(2t_1))^{2r}$$ which implies
\begin{eqnarray*}
(a(t))^{-r}\me (\max_{0\leq k\leq 2^{I}-1}J_{22}(t,k))^{2r}&\leq&
(a(t))^{-r}2^I \max_{0\leq k\leq 2^I-1}\me
(J_{22}(t,k))^{2r}\\&\leq& (a(t))^{-r}2^I \me (\nu(1))^{2r}
(a(2t_1))^{2r}.
\end{eqnarray*}
The right-hand side converges to zero as $t\to\infty$ by our
choice of $r$. Consequently, $(a(t))^{-1/2} \max_{0\leq k\leq
2^{I}-1}J_{22}(t,k)\tp 0$ as $t\to\infty$ by Markov's inequality.

Using a counterpart of the first inequality in \eqref{mom} for the
martingale $(R^\ast(l,t), \mathcal{F}_l)_{l\in\mn_0}$, where
$R^\ast(0,t):=0$ and
$$R^\ast(l,t):=\sum_{j=0}^{l-1}\big(\1_{\{vt<S_j+\eta_{j+1}\leq
ut\}}-(F(ut-S_j)-F(vt-S_j))\big)\1_{\{S_j\leq vt\}},\quad
l\in\mn$$ for $u=(k+1)2^{-I}t$ and $v=k2^{-I}t$, one can check
that
\begin{eqnarray*}
\me (J_{21}(t,k))^{2r}&\leq& C_r\bigg(\me
\bigg(\int_{[0,\,k2^{-I}t]}(F((k+1)2^{-I}t-y)-F(k2^{-I}t-y)){\rm
d}\nu(y)\bigg)^r \\&+&\me
\int_{[0,\,k2^{-I}t]}(F((k+1)2^{-I}t-y)-F(k2^{-I}t-y)){\rm
d}\nu(y)\bigg).
\end{eqnarray*}
In view of \eqref{11b} the right-hand side does not exceed
$$C_r(\me (\nu(1))^r(a(2^{-I}t))^r+\me \nu(1)a(2^{-I}t))\leq C_r(\me (\nu(1))^r(a(2t_1))^r+\me \nu(1)a(2t_1)).$$ Arguing as above we conclude that
$(a(t))^{-1/2}\max_{0\leq k\leq 2^{I}-1}J_{21}(t,k)\tp 0$ as
$t\to\infty$, and \eqref{33} for $i=2$ follows. The proof of
Theorem \ref{main1} is complete.

\section{Proof of Theorem \ref{main}}\label{pr2}

Set $f(t):=\sqrt{t(1-F(t))}$ for $t>0$. In view of \eqref{regular}
\begin{equation}\label{kar}
\sqrt{\int_0^t(1-F(y)){\rm
d}y}~\sim~(1-\beta)^{-1/2}t^{1/2-\beta/2}(\ell(t))^{1/2}~\sim~(1-\beta)^{-1/2}f(t)
\end{equation}
as $t\to\infty$. Assuming that $\me \xi^r<\infty$ for some
$r>2(1-\beta)^{-1}$ we intend to show that
$$\dfrac{\sup_{0\leq u\leq T}\big|\sum_{k\geq 0}(1-F(ut-S_k))\1_{\{S_k\leq ut\}}-\mu^{-1}\int_0^{ut}(1-F(y)){\rm d}y\big|}
{f(t)} ~\tp~ 0,\quad t\to\infty$$ for any $T>0$. This in
combination with \eqref{kar} and Theorem \ref{main1} is sufficient
for the proof of the $J_1$-convergence.

We proceed by observing that $$\sum_{k\geq
0}(1-F(t-S_k))\1_{\{S_k\leq t\}}-\mu^{-1}\int_0^t(1-F(y)){\rm
d}y=\int_{[0,\,t]}(1-F(t-y)){\rm d}(\nu(y)-\mu^{-1}y).$$
Integration by parts yields
\begin{eqnarray*}
&&\int_{[0,\,t]}(1-F(t-y)){\rm
d}(\nu(y)-\mu^{-1}y)+\Prob\{\xi=t\}\\&=&\nu(t)-\mu^{-1}t-\int_{[0,\,t)}(\nu(t-y)-\mu^{-1}(t-y)){\rm
d}F(y)=\bigg(\nu(t)-\mu^{-1}t-\sigma\mu^{-3/2}W(t)\\&-&\int_{[0,\,t)}(\nu(t-y)-\mu^{-1}(t-y)-\sigma\mu^{-3/2}W(t-y)){\rm
d}F(y)\bigg)\\&+&\sigma\mu^{-3/2}\bigg(W(t)-\int_{[0,\,t)}W(t-y){\rm
d}F(y)\bigg)=: R_1(t)+\sigma\mu^{-3/2}R_2(t),
\end{eqnarray*}
where $\sigma^2={\rm Var}\,\xi<\infty$ and $W$ is a standard
Brownian motion as defined in Lemma \ref{chs}. For any $T>0$
\begin{eqnarray*}
\sup_{0\leq u\leq T}|R_1(ut)|&\leq& \sup_{0\leq u\leq
T}|\nu(ut)-\mu^{-1}ut-\sigma\mu^{-3/2}W(ut)|\\&+&\sup_{0\leq u\leq
T}\int_{[0,\,ut)}|\nu(ut-y)-\mu^{-1}(ut-y)-\sigma\mu^{-3/2}W(ut-y)|{\rm
d}F(y)\\&\leq& \sup_{0\leq u\leq
Tt}|\nu(u)-\mu^{-1}u-\sigma\mu^{-3/2}W(u)|\\&+&\sup_{0\leq u\leq
T}\sup_{0\leq y\leq
ut}|\nu(y)-\mu^{-1}y-\sigma\mu^{-3/2}W(y)|\\&\leq& 2\sup_{0\leq
u\leq Tt}|\nu(u)-\mu^{-1}u-\sigma\mu^{-3/2}W(u)|.
\end{eqnarray*}
By Lemma \ref{chs} the right-hand side is $o(t^{1/r})$ a.s.\ as
$t\to\infty$. Hence, our choice of $r$ in combination with
\eqref{kar} ensure that
$$\lit \frac{\sup_{0\leq u\leq
T}|R_1(ut)|}{f(t)}=0\quad \text{a.s.}$$ Further, we note that
$$R_2(t)=W(t)(1-F(t))+\int_{[0,\,t]}(W(t)-W(t-y)){\rm
d}F(y)=:R_{21}(t)+R_{22}(t).$$ Pick now $\varepsilon\in
(0,(1-\beta)/2)$ if $\beta\in [1/2,1)$ and $\varepsilon\in
(0,1/2-\beta)$ if $\beta\in [0,1/2)$. With this $\varepsilon$, we
have for any $T>0$
\begin{eqnarray*}
\sup_{0\leq u\leq T}\,|R_{22}(ut)|&\leq& \sup_{0\leq u\leq T}\,
\int_{[0,\,ut]}\frac{|W(ut)-W(ut-y)|}{y^{1/2-\varepsilon}}y^{1/2-\varepsilon}{\rm
d}F(y)\\&\leq&\sup_{0\leq u\leq T} \sup_{0\leq x\leq
ut}\,\frac{|W(ut)-W(ut-x)|}{x^{1/2-\varepsilon}}
\int_{[0,\,ut]}y^{1/2-\varepsilon}{\rm d}F(y)\\&\leq& \sup_{0\leq
v<u\leq tT}\,\frac{|W(u)-W(v)|}{(u-v)^{1/2-\varepsilon}}
\int_{[0,\,Tt]}y^{1/2-\varepsilon}{\rm d}F(y)\\&\od& \sup_{0\leq
v<u\leq T}\,\frac{|W(u)-W(v)|}{(u-v)^{1/2-\varepsilon}}
t^\varepsilon\int_{[0,\,Tt]}y^{1/2-\varepsilon}{\rm d}F(y).
\end{eqnarray*}
Here, $$\sup_{0\leq v<u\leq
T}\,\frac{|W(u)-W(v)|}{(u-v)^{1/2-\varepsilon}}<\infty\quad\text{a.s.}$$
because the Brownian motion $W$ is locally H\"{o}lder continuous
with exponent $1/2-\varepsilon$ (for any $\varepsilon\in
(0,1/2)$), and the distributional equality denoted by $\od$ is a
consequence of self-similarity of $W$ with index $1/2$. Now it is
convenient to treat two cases separately.

\noindent {\sc Case $\beta\in [1/2,1)$} in which
$$\dfrac{t^\varepsilon\int_{[0,\,Tt]}y^{1/2-\varepsilon}{\rm d}F(y)}{f(t)}~\sim~\frac{\me\eta^{1/2-\varepsilon}}
{t^{1/2-\beta/2-\varepsilon}(\ell(t))^{1/2}}\to 0,\quad
t\to\infty$$ by \eqref{kar} and our choice of $\varepsilon$. This
proves
\begin{equation}\label{100}
\dfrac{\sup_{0\leq u\leq T}\,|R_{22}(ut)|}{f(t)}~\tp~ 0,\quad
t\to\infty.
\end{equation}

\noindent {\sc Case $\beta\in [0, 1/2)$}. Here, we conclude that
$$\dfrac{t^\varepsilon\int_{[0,\,Tt]}y^{1/2-\varepsilon}{\rm d}F(y)}{f(t)}~\sim~
\dfrac{T^{1/2-\beta-\varepsilon}(\ell(t))^{1/2}}{(1/2-\beta-\varepsilon)t^{\beta/2}}~\to~0,\quad
t\to\infty$$ having utilized \eqref{kar}, Theorem 1.6.4 in
\cite{BGT} which is applicable by our choice of $\varepsilon$ and
the fact that $\lit \ell(t)=0$ when $\beta=0$. Thus, \eqref{100}
holds in this case, too.

It remains to check weak convergence\footnote{Weak convergence on
$D(0,\infty)$ follows immediately from the fact that $\lit
(1-F(ut))/(1-F(t))=u^{-\beta}$ locally uniformly in $u$ on
$(0,\infty)$. A longer proof is needed to treat weak convergence
on $D[0,\infty)$, i.e., with $0$ included.} on $D$ of
$R_{21}(\cdot t)/f(t)$ to the zero function or equivalently
\begin{equation}\label{101}
\dfrac{\sup_{0\leq u\leq T}\,|R_{21}(ut)|}{f(t)}~\tp~0,\quad
t\to\infty
\end{equation}
for each $T>0$. We shall only consider the case where $T>1$, the
case $T\in (0,1]$ being analogous and simpler. By Potter's bound
(Theorem 1.5.6 (iii) in \cite{BGT}), for any chosen $A>1$ and
$\delta>0$ there exists $t_0>0$ such that $1-F(ut)/(1-F(t))\leq
Au^{-\beta-\delta}$ whenever $u\in (0,1]$ and $ut\geq t_0$. With
this $t_0$, write
$$\sup_{0\leq u\leq T}\,|R_{21}(ut)|\leq \sup_{0\leq u\leq
t_0/t}\,|R_{21}(ut)|\vee \sup_{t_0/t\leq u\leq
1}\,|R_{21}(ut)|\vee \sup_{1\leq u\leq T}\,|R_{21}(ut)|.$$

For the first supremum on the right-hand side we have $\sup_{0\leq
u\leq t_0/t}\,|W(ut)|(1-F(ut))\leq \sup_{0\leq u\leq t_0}\,|W(u)|$
which converges to zero a.s.\ when divided by $f(t)$.

For the third supremum,
\begin{eqnarray*}
\sup_{1\leq u\leq T}\,|W(ut)|(1-F(ut))&\leq& (1-F(t))\sup_{0\leq
u\leq T}\,|W(ut)|\\&\od& t^{1/2}(1-F(t))\sup_{0\leq u\leq
T}\,|W(u)|,
\end{eqnarray*}
and the right hand-side divided by $f(t)$ converges to zero a.s.
in view of \eqref{kar}.

Finally,
\begin{equation}\label{102}
\frac{\sup_{t_0/t\leq u\leq 1}\,|W(ut)|(1-F(ut))}{1-F(t)}\leq
A\sup_{t_0/t\leq u\leq 1}\,|W(ut)|u^{-\beta-\delta}.
\end{equation}
As before we distinguish the two cases.

\noindent {\sc Case $\beta\in [1/2,1)$}. Choose $\delta$
satisfying $\delta\in (0,(1-\beta)/2)$. The law of the iterated
logarithm for $|W|$ at large times guarantees that $\lit\,
|W(t)|t^{-\beta-\delta}=0$ a.s.\ and thereupon $\sup_{u\geq
t_0}\,|W(u)|u^{-\beta-\delta}<\infty$ a.s. With this at hand we
continue \eqref{102} as follows:
\begin{eqnarray*}
\frac{\sup_{t_0/t\leq u\leq 1}\,|W(ut)|(1-F(ut))}{f(t)}&\leq&
\dfrac{At^{\beta+\delta}\sqrt{1-F(t)}\sup_{t_0\leq u\leq
t}\,\big(|W(u)|u^{-\beta-\delta}\big)}{t^{1/2}}\\&\sim&
\dfrac{A\sup_{u\geq
t_0}\,\big(|W(u)|u^{-\beta-\delta}\big)(\ell(t))^{1/2}}{t^{1/2-\beta/2-\delta}}\quad\text{a.s.}
\end{eqnarray*}
having utilized \eqref{kar} for the last asymptotic equivalence.
The right-hand side converges to zero a.s.

\noindent {\sc Case $\beta\in [0,1/2)$}. Pick $\delta$ so small
that $\beta+\delta<1/2$. The law of the iterated logarithm for
$|W|$ at small times entails $\lim_{t\to 0+}\,
|W(t)|t^{-\beta-\delta}=0$ a.s.\ whence $\sup_{0\leq u\leq
1}\,|W(u)|u^{-\beta-\delta}<\infty$ a.s. Continuing \eqref{102}
with the help of self-similarity of $W$ we further infer
$$\frac{\sup_{t_0/t\leq u\leq 1}\,|W(ut)|(1-F(ut))}{f(t)}\leq
A\sup_{0\leq u\leq 1}\,|W(u)|u^{-\beta-\delta}\sqrt{1-F(t)}.$$ It
remains to note that the right-hand side trivially converges to
zero a.s.

Combining pieces together we conclude that \eqref{101} holds. The
proof of Theorem \ref{main} is complete.

\section{Integral representation of the limit process
$V_\beta$}\label{proc}

First of all, we note that $V_0$ is a standard Brownian motion.
Therefore, throughout the rest of the section we assume that
$\beta\in (0,1)$.

Denote by $B:=(B(u,v))_{u,v\geq 0}$ a standard Brownian sheet,
i.e., a two-parameter continuous centered Gaussian field with $\me
B(u_1,v_1)B(u_2,v_2)=(u_1\wedge u_2)(v_1\wedge v_2)$. In
particular, $B$ is a Brownian motion in $u$ (in $v$) for each
fixed $v$ ($u$). See Section 3 in \cite{Walsh:1986} for more
properties of $B$. It turns out that the limit process $V_\beta$
can be represented as the integral of a deterministic function
with respect to the Brownian sheet. Such integrals are constructed
in \cite{Ito:1951}. Also, these can be thought of as particular
instances of the integrals of the first kind with respect to the
Brownian sheet, see Section 4 in \cite{Walsh:1986}. Set
\begin{equation}\label{equ}
V^\ast_\beta(u)=\sqrt{1-\beta}\int_{[0,\,u]}\int_{[0,\infty)}\1_{\{x+z^{-1/\beta}>u\}}{\rm
d}B(x,z),\quad u\geq 0.
\end{equation}
Clearly, the process $V_\beta^\ast:=(V_\beta^\ast(u))_{u\geq 0}$
is centered Gaussian. Since
\begin{eqnarray*}
&&\me
V^\ast_\beta(u)V^\ast_\beta(s)\\&=&(1-\beta)\int_{[0,\infty)}\int_{[0,\infty)}\1_{\{x+z^{-1/\beta}>u\}}\1_{[0,\,u]}(x)
\1_{\{x+z^{-1/\beta}>s\}}\1_{[0,\,s]}(x){\rm d}z{\rm
d}x\\&=&(1-\beta)\int_0^s\int_0^\infty
\1_{\{x+z^{-1/\beta}>u\}}{\rm d}z{\rm d}x=(1-\beta)\int_0^s
(u-x)^{-\beta}{\rm d}x\\&=&u^{1-\beta}-(u-s)^{1-\beta}
\end{eqnarray*}
for $0\leq s\leq u$, we conclude that $V_\beta^\ast$ is a version
of $V_\beta$.

The discussion above does not give a clue on where equality
\eqref{equ} comes from. Here is a non-rigorous argument based on
the idea from \cite{Krichagina+Puhalskii:1997} which allows one to
guess \eqref{equ}. We start with an integral representation
\begin{eqnarray}
&&\frac{\sum_{k\geq 0}\big(\1_{\{S_k \leq ut <
S_k+\eta_{k+1}\}}-(1-F(ut-S_k))\1_{\{S_k\leq
ut\}}\big)}{\sqrt{\mu^{-1}\int_0^t (1-F(y)){\rm
d}y}}\notag\\&=&\int_{[0,\,u]}\int_{[0,\infty)}\1_{\{x+z>u\}}{\rm
d}\bigg(\frac{\sum_{k=1}^{\nu(xt)}\1_{\{\eta_k\leq
zt\}}-\nu(xt)F(zt)}{\sqrt{\mu^{-1}\int_0^t (1-F(y)){\rm
d}y}}\bigg)\label{equ1}
\end{eqnarray}
where $\nu(t)=\inf\{k\in\mn: S_k>t\}$ for $t\geq 0$. It is likely
that
$$\frac{\sum_{k=1}^{[xt]}\1_{\{\eta_k\leq zt\}}-[xt]
F(zt)}{\sqrt{t(1-F(t))}}$$ converges weakly as $t\to\infty$ to
$B(x,z^{-\beta})$ on some appropriate space of functions
$g:[0,\infty)\times [0,\infty))\to\R$ equipped with some topology
which is strong enough to ensure continuity of composition. The
latter together with \eqref{regular} and the well-known relation
$t^{-1}\nu(tx)\overset{{\rm J_1}}{\Rightarrow} \mu^{-1}x$ as
$t\to\infty$ should entail that
$$\frac{\sum_{k=1}^{\nu(xt)}\1_{\{\eta_k\leq zt\}}-\nu(xt)
F(zt)}{\sqrt{\mu^{-1}\int_0^t (1-F(y)){\rm d}y}}$$ converges
weakly to $\sqrt{1-\beta} B(x,z^{-\beta})$. One may expect that
the right-hand side of \eqref{equ1} converges weakly to the
right-hand side of \eqref{equ}. On the other hand, the left-hand
side of \eqref{equ1} converges weakly to $V_\beta$ by Theorem
\ref{main1}.

\section{Appendix}

The following result can be found in the proof of Lemma 7.3 in
\cite{Alsmeyer+Iksanov+Marynych:2016}.
\begin{Lemma}\label{impo1}
Let $G:[0,\infty)\to [0,\infty)$ be a locally bounded function.
Then, for any $l\in\N$
\begin{equation}\label{impo1-gen}
\E \bigg(\sum_{k\geq 0}G(t-S_k)\1_{\{S_k\leq t\}}\bigg)^l\leq
\bigg(\sum_{j=0}^{[t]}\sup_{y\in[j,\,j+1)}G(y)\bigg)^l\me
(\nu(1))^l,\quad t\geq 0.
\end{equation}
\end{Lemma}
The second auxiliary result is well-known. See, for instance,
Theorem 2.1 (b) in \cite{Iksanov+Meiners:2010}. It is of principal
importance here that $\xi$ is a.s.\ positive rather than
nonnegative.
\begin{Lemma}\label{fin}
For all $a>0$ and all $t>0$ $\me e^{a\nu(t)}<\infty$.
\end{Lemma}
Also, we need a classical strong approximation result, see
Corollary 3.1 (ii) in \cite{Csorgo+Horvath+Steinebach:1987}.
\begin{Lemma}\label{chs}
Suppose that $\me \xi^r<\infty$ for some $r>2$. Then there exists
a standard Brownian motion $W$ such that
$$\lit t^{-1/r}\sup_{0\leq s\leq
t}\,\big|\nu(s)-\mu^{-1}s-\sigma\mu^{-3/2}W(s)\big|=0\quad\text{{\rm
a.s.}},$$ where $\mu=\me\xi$ and $\sigma^2={\rm Var}\,\xi$.
\end{Lemma}

%
%
%

\end{document}